\documentclass[11pt]{article}
\usepackage{indentfirst}
\usepackage{latexsym}
\usepackage{hyperref}
\usepackage{graphicx}
\usepackage{mathrsfs}
\usepackage{bookmark}
\usepackage{enumerate}
\usepackage{amsfonts,amsthm,amssymb,mathtools}
\usepackage{amsmath}
\usepackage{latexsym, euscript, epic, eepic, color}
\usepackage{amssymb}

\usepackage{enumerate}
\usepackage[all]{xy}
\usepackage{setspace}
\allowdisplaybreaks

\newenvironment{wst}
{\setlength{\leftmargini}{1.5\parindent}
	\begin{itemize}
		\setlength{\itemsep}{-1.1mm}}
	{\end{itemize}}

\usepackage{epstopdf}
\numberwithin{equation}{section}
\newtheorem{theorem}{Theorem}[section]
\newtheorem{corollary}[theorem]{Corollary}
\newtheorem{definition}[theorem]{Definition}

\newtheorem{lemma}[theorem]{Lemma}

 \allowdisplaybreaks

\begin{document}
\baselineskip=16pt

\title{Signless Laplacian spectral radius and fractional matchings in graphs\footnote{This work was supported by the National Natural Foundation of China [61773020]. Corresponding author: Yingui Pan(panygui@163.com)
 }
}

\author{Yingui Pan$^{a}$,   Jianping Li$^{a}$   \\
	\small  $^{a}$College of Liberal Arts and Sciences, National University of Defense Technology, \\
	\small  Changsha, China, 410073.\\
}

\date{\today}

\maketitle

\begin{abstract}
A fractional matching of a graph $G$ is a function $f$ giving each edge a number in $[0,1]$ such that $\sum_{e\in\Gamma(v)}f(e)\leq1$ for each vertex $v\in V(G)$, where $\Gamma(v)$ is the set of edges incident to $v$. The fractional matching number of $G$, written $\alpha^{\prime}_*(G)$, is the maximum value of $\sum_{e\in E(G)}f(e)$ over all fractional matchings. In this paper, we investigate the relations between the fractional matching number and the signless Laplacian spectral radius of a graph. Moreover, we give some sufficient spectral conditions for the  existence of a fractional perfect matching.
\end{abstract}



\section{Introduction}
Graphs considered in this paper are simple and undirected. Let $G$ be a graph with vertex set $V(G)$ and edge set $E(G)$ such that $|V(G)|=n$ and $|E(G)|=m$. As usual, $d(u)$ stands for the \textit{degree} of a vertex $u$ in $G$.  The \textit{adjacent matrix} of $G$ is $A(G)=(a_{ij})_{n\times n}$, where $a_{ij}=1$ if $i$ and $j$ are adjacent, and $a_{ij}=0$ otherwise. The \textit{diagonal matrix} of $G$ is $D(G)=(d(i))_{n\times n}$, where $d(i)$ is the degree of vertex $i$. Let $\lambda_1(G)\geq\lambda_2(G)\geq\cdots\geq\lambda_n(G)$, $\mu_1(G)\geq\mu_2(G)\geq\ldots\geq\mu_n(G)$ and $q_1(G)\geq q_2(G)\geq\ldots\geq q_n(G)$ be the eigenvalues of $A(G)$, $L(G)$ and $Q(G)$, respectively, where $L(G)=D(G)-A(G)$ and $Q(G)=D(G)+A(G)$. Particularly, the eigenvalues $\lambda_1(G)$, $\mu_1(G)$ and $q_1(G)$ are called the \textit{spectral radius},  \textit{Laplacian spectral radius} and \textit{signless Laplacian spectral radius} of $G$, respectively.  For a set $S\subseteq V(G)$, let $G[S]$ denote the subgraph of $G$ induced by $S$, and let $G-S$ be the graph obtained from $G$ by deleting the vertices in $S$ together with their incident edges.  The \textit{complement graph} $G^c$ of $G$ is the graph whose vertex set is $V(G)$ and whose edges are the pairs of nonadjacent vertices of $G$. For any two vertex-disjoint graphs $G_1$ and $G_2$, we use $G_1\vee G_2$ to denote the \textit{join} of graphs $G_1$ and $G_2$ and $G_1\cup G_2$ to denote the \textit{disjoint union} of graphs $G_1$ and $G_2$.

An edge set $M$ of $G$ is called a \textit{matching} if any two edges in $M$ have no common vertices. If each vertex of $G$ is incident with exactly one edge of $M$, then $M$ is called a \textit{perfect matching} of $G$. The \textit{matching number} of a graph $G$, denoted by $\alpha^{\prime}(G)$, is the number of edges in a maximum matching. A fractional matching of a graph $G$ is a function $f$ giving each edge a number in $[0,1]$ such that $\sum_{e\in\Gamma(v)}f(e)\leq1$ for each vertex $v\in V(G)$, where $\Gamma(v)$ is the set of edges incident to $v$. The fractional matching number of $G$, written $\alpha^{\prime}_*(G)$, is the maximum value of $\sum_{e\in E(G)}f(e)$ over all fractional matchings $f$. A \textit{fractional perfect matching} of a graph $G$ is a fractional matching $f$ with $\alpha^{\prime}_*(G)=\sum_{e\in E(G)}f(e)=\frac{n}{2}$, and a fractional perfect matching $f$ of a graph $G$ is a perfect matching if it takes only the values 0 or 1.


Fractional matching has attracted many researchers' attention.  Behrend et al. \cite{beh} established a lower bound on the fractional matching number of a graph with given some graph parameters and characterized the graphs whose fractional matching number attains the lower bound. Choi et al. \cite{cho} gave the tight upper bounds on the difference and ratio of the fractional matching number and matching number among all $n$-vertex graphs, and  characterized the infinite family of graphs where equalities hold.  O \cite{o} investigated the relations between the spectral radius of a connected graph with minimum degree $\delta$ and its fractional matching number, and gave a lower bound on the fractional matching number in terms of the spectral radius and minimum degree. Xue \cite{xue} studied the connections between the fracional matching number and the Laplacian spectral radius of a graph, and obtained some lower bounds on the fractional matching number of a graph. Moreover, they presented some sufficient spectral conditions for the existence of a fractional perfect matching.

 Motivated by \cite{o,xue}, we investigate the relations between the signless Laplacain spectral radius of a graph and its fractional matching number. In Section 2, we list some useful lemmas. In Section 3, we establish a lower bound on the fraction number of a graph in terms of its signless Laplacian spectral radius and minimum degree. In Section 4, we obtain some sufficient spectral conditions for the existence of a fractional perfect matching.

\section{Preliminaries}
In this section, we list some lemmas which will be used in our paper later. Some fundamental properties of fractional matching were obtained in \cite{sch}.
\begin{lemma}\cite{sch}\label{lemm0}
For any graph $G$, let $\alpha^{\prime}_*(G)$ be the fractional matching number of $G$. Then
\begin{wst}
	\item[{\rm (i)}] $2\alpha^{\prime}_*(G)$ is an integer.
	\item[{\rm (ii)}]
	 $\alpha^{\prime}_*(G)=\frac{1}{2}(n-\max\{i(G-S)-|S|\})$, where the maximum is taken over all $S\subseteq V(G)$.
\end{wst}
\end{lemma}

\begin{lemma}\cite{shen}\label{lemmm1}
	Let $G$ be a connected graph. If $H$ is a subgraph of $G$, then $q_1(H)\leq q_1(G)$.
\end{lemma}

\begin{lemma}\cite{cve1}\label{k1}
	Let $K_n$ be a complete graph of order $n$, where $n\geq2$. Then $q_1(K_n)=2n-2$.
	\end{lemma}
We now explain the concepts of the equitable matrix and equitable partition.
\begin{definition}\cite{bre}\label{def1}
Let $M$ be a real matrix of order $n$ described in the following block form
$$M=\left(\begin{array}{ccc}
M_{11}&  \cdots& M_{1t} \\
\vdots&  \ddots&  \vdots\\
M_{t1}& \cdots&  M_{tt}
\end{array}
\right),$$
where the blocks $M_{ij}$ are $n_i\times n_j$ matrices for any $1\leq i,j\leq t$ and $n=n_1+\ldots+n_t$. For $1\leq i,j\leq t$, let $b_{ij}$ denote the average row sum of $M_{ij}$, i.e. $b_{ij}$ is the sum of all entries in $M_{ij}$ divided by the number of rows. Then $B(M)=(b_{ij})$ (simply by $B$) is called the quotient matrix of $M$. If for each pair $i,j$, $M_{ij}$ has constant row sum, then $B$ is called the equitable quotient matrix of $M$ and the partition is called equitable.
\end{definition}
\begin{lemma}\cite{bre}\label{le1}
	Let $M$ be a symmetric real matrix. If $M$ has an equitable partition and $B$ is the corresponding matrix, then each eigenvalue of $B$ is also an eigenvalue of $M$.
\end{lemma}

 The relation  between $\lambda_1(B)$ and $\lambda_1(M)$ is obtained as below.

\begin{lemma}\cite{you}\label{lemm1}
	Let $B$ be the equitable matrix of $M$ as defined in Definition \ref{def1}, and $M$ be a nonnegative matrix. Then $\lambda_1(B)=\lambda_1(M)$.
\end{lemma}
 O \cite{o} constructed a family of connected bipartite graphs $\mathcal{H}(\delta,k)$, where  $\delta$ and $k$ are two  positive integers. For each graph $G\in\mathcal{H}(\delta,k)$ with the bipartition $V(G)=V_1\cup V_2$,  $G$ satisfies the following  conditions:
\begin{wst}
	\item[{\rm (i)}] every vertex in $V_1$ has degree $\delta$,
	\item[{\rm (ii)}] $|V_1|=|V_2|+k$,
	\item[{\rm (iii)}]  the degrees of vertices in $V_2$ are equal.
\end{wst}


The exact values of the fractional matching number and the  spectral radius for graphs in $\mathcal{H}(\delta,k)$ are obtained as below.
\begin{lemma}\cite{o}\label{lem1}
If $H\in\mathcal{H}(\delta,k)$, then $\alpha^{\prime}_*(H)=\frac{|V(H)|-k}{2}$ and $\lambda_1(H)=\delta\sqrt{1+\frac{2k}{|V(H)|-k}}$.
\end{lemma}


We now determine the signless Laplacian spectral radius of graphs in $\mathcal{H}(\delta,k)$.

\begin{lemma}\label{l2}
If $H\in\mathcal{H}(\delta,k)$, then  $q_1(H)=\frac{2\delta|V(H)|}{|V(H)|-k}$.
\end{lemma}
\begin{proof}
If $H\in\mathcal{H}(\delta,k)$, by the partition $V(H)=V_1\cup V_2$, we can obtain the quotient matrix of $Q(H)$:
	$$B=\left(\begin{array}{cc}
\delta& 	\delta\\
	\frac{\delta|V_1|}{|V_2|}&  	\frac{\delta|V_1|}{|V_2|}
	\end{array}
	\right).$$
It is easy to calculate that  $\lambda_1(B)=\delta\big(1+\frac{|V_1|}{|V_2|}\big)=\delta\big(2+\frac{k}{|V_2|}\big)$. By the construction of $\mathcal{H}(\delta,k)$, the partition $V(H)=V_1\cup V_2$ is equitable and $|V_2|=\frac{|V(H)|-k}{2}$. By Lemma \ref{lemm1}, we have
	$$q_1(H)=\lambda_1(Q(H))=\lambda_1(B)=\delta\bigg(2+\frac{k}{|V_2|}\bigg)=\delta\bigg(2+\frac{2k}{|V(H)|-k}\bigg)=\frac{2\delta|V(H)|}{|V(H)|-k}$$
	as desired.
\end{proof}

\section{A relationship between $q_1(G)$ and $\alpha^{\prime}_{*}(G)$}
 In this section, we investigate the relationship between the signless Laplacian spectral radius of a graph with minimum degree $\delta$ and its fractional matching number. Similar to the proof of Lemma 3.2 in \cite{o}, we can obtain the following lemma.
\begin{lemma}\label{lemmm2}
Let $G$ be an $n$-vertex connected graph with minimum degree $\delta$, and let $k$ be a real number between 0 and $n$. If $q_1(G)<\frac{2n\delta}{n-k}$, then $\alpha^{\prime}_{*}(G)>\frac{n-k}{2}$.
\end{lemma}

\begin{proof}
If $\alpha^{\prime}_{*}(G)\leq\frac{n-k}{2}$, by Lemma \ref{lemm0}, there exists a vertex set $S\subseteq V(G)$ such that  $i(G-S)-|S|\geq{k}$. Since $i(G-S)$ is an integer, then $i(G-S)-|S|\geq{\lceil{k}\rceil}$. Let $A$ be the set of all isolated vertices in $G-S$.  Then,
$$|A|=i(G-S)\geq|S|+\lceil{k}\rceil.$$
 Consider the bipartite subgraph $H$ with the partitions $V(H)=A\cup S$ such that $E(H)$ is the set of edges of $G$ having one endpoint in $A$ and the other in $S$. Let $r$ be the number of edges in $H$. Then $r\geq\delta|A|$. For the partition $V(H)=A\cup S$, we can obtain a quotient matrix of $Q(H)$ as below:
	$$B=\left(\begin{array}{cc}
	\frac{r}{|A|}& 	\frac{r}{|A|}\\
	\frac{r}{|S|}&  	\frac{r}{|S|}
	\end{array}
	\right).$$
	It is easy to calculate that $\lambda_1(B)=\frac{r(|A|+|S|)}{|A||S|}$. Since the partition is equitable, by Lemma \ref{le1}, we have
	$$q_1(G)=\lambda_1(Q(G))\geq\lambda_1(B)=\frac{r(|A|+|S|)}{|A||S|}\geq\delta\frac{|A|+|S|}{|S|}\geq\delta\frac{2|S|+\lceil{k}\rceil}{|S|}\geq\delta\bigg(2+\frac{2\lceil{k}\rceil}{n-k}\bigg)\geq\frac{2n\delta}{n-k}$$
since $r\geq\delta{|A|}$, $|A|\geq |S|+\lceil{k}\rceil$,  $n\geq |A|+|S|\geq2|S|+k$ and $|S|\geq\delta$.
\end{proof}

 \begin{theorem}\label{t1}
 	If $G$ is an $n$-vertex graph with minimum degree $\delta$, then we have $$\alpha^{\prime}_*(G)\geq\frac{n\delta}{q_1(G)},$$ with equality if and only if $k=\frac{n(q_1(G)-2\delta)}{q_1(G)}$ is an integer and $G$ is an element of $\mathcal{H}(\delta,k)$.
 \end{theorem}

\begin{proof}
	By Lemma \ref{lemmm2}, $\alpha^{\prime}_{*}(G)>\frac{n-k}{2}$ if $q_1(G)<\frac{2n\delta}{n-k}$. Note that $\frac{2n}{n-k}$ is an increasing function of $k$ on $[0,n)$, thus $\frac{2n\delta}{n-k}$ decreases towards $q_1(G)$ as $k$ decreases towards $z$, where $z=\frac{n(q_1(G)-2\delta)}{q_1(G)}$.  Then for each value $k\in (z,n)$, we have $\alpha^{\prime}_{*}(G)>\frac{n-k}{2}$ by Lemma \ref{lemmm2}. Let $k$ tend to $z$ and finally equal to $z$, we obtain $\alpha^{\prime}_*(G)\geq\frac{n\delta}{q_1(G)}$ as desired.
	
	If $k=\frac{n(q_1(G)-2\delta)}{q_1(G)}$ is an integer and $G\in\mathcal{H}(\delta,k)$, then by Lemma \ref{lem1}, we have $\alpha^{\prime}_{*}(G)=\frac{n-k}{2}=\frac{n\delta}{q_1(G)}$.  For the 'only if' part, assume that $\alpha^{\prime}_{*}(G)=\frac{n\delta}{q_1(G)}$. Then $k=z$ and the inequalities in Lemma \ref{lemmm2} become equality. Since $\lceil{k}\rceil=k$, $k$ must be an integer. In addition, note that $r=\delta{|A|}$, $|A|=|S|+k$,  $n=2|S|+k$ and $|S|=\delta$, $G$ must be included in $\mathcal{H}(\delta,k)$.
\end{proof}

Let $G$ be a bipartite  graph with partition $V(G)=V_1\cup V_2$.  Then $G$ is said to be semi-regular if all vertices in $V_i$ have the same degree $d_i$ for $i=1,2$.

\begin{lemma}\cite{cve}\label{lemm4}
Let $G$ be a connected graph graph. Then $$q_1(G)\leq\max\{d(u)+d(v):uv\in E(G)\},$$
with equality if and only if $G$ is a regular bipartite graph or a semi-regular bipartite graph.
\end{lemma}

Let  $g(G)$ be the length of a shortest cycle in $G$, and let $\alpha(G)$ be the independence number of $G$ which is the cardinality of the largest independent set of $G$. Similar to the proof of Theorem 2.6 in \cite{xue}, we can obtain the following theorem.
\begin{theorem}\label{lemm5}
Let $G$ be a graph with independence number $\alpha(G)$. If $g(G)\geq5$, then $q_1(G)<2+\alpha(G)$.
\end{theorem}

\begin{proof}
	Without loss of generality, assume that $G$ is connected and $d(u_1)+d(v_1)=\max\{d(u)+d(v):uv\in E(G)\}$. Let $A=N(u_1)\setminus\{v_1\}$ and $B=N(v_1)\setminus\{u_1\}$. Since $g(G)\geq5$, then $|A|+|B|\leq\alpha(G)$ and thus $d(u_1)+d(v_1)=2+|A|+|B|\leq2+\alpha(G)$. By Lemma \ref{lemm4}, $q_1(G)\leq2+\alpha(G)$. If $q_1(G)=2+\alpha(G)$, then $\alpha(G)=|A|+|B|$ and thus $G$ is bipartite regular or semi-regular. Suppose that $|A|\geq|B|$ for convenience. Let $w_1$ be a vertex of $B$. Then $u_1,w_1\in B$ since both $u_1$ and $w_1$ are adjacent to $v_1$.  Since $G$ is regular or semi-regular, then $d(u_1)=d(w_1)$ and thus $|N(w_1)\setminus\{v_1\}|=|A|$. Note that $N(w_1)\cup A$ is an independent set of $G$, then $\alpha(G)\geq|N(w_1)\cup A|=2|A|+1$, a contradiction to the fact $\alpha(G)=|A|+|B|$.
\end{proof}

Together with Theorems \ref{t1} and \ref{lemm5}, we obtain a lower bound on the fractional matching number in terms of the independence number and minimum degree, which improves the lower bound obtained in \cite{xue}.

\begin{corollary}\label{lem277}
Let $G$ be a connected graph with independence number $\alpha(G)$ and minimum degree $\delta$. If $g(G)\geq5$, then
$$\alpha^{\prime}_*(G)>\frac{n\delta}{\alpha(G)+2}.$$
\end{corollary}



\section{Signless Laplacian spectral radius and fractional perfect matching}
Some sufficient condition for the existence of a fractional perfect matching in a graph in terms of the spectral radius were obtain in \cite{xue}. In this section, we are devoted to give some  sufficient conditions for a graph to have a fractional perfect matching from the viewpoint of signless Laplacian spectral radius.

\begin{theorem}
Let $G$ be an $n$-vertex connected graph with minimum degree $\delta$. If $q_1(G)<\frac{2n\delta}{n-1}$, then $G$ has a fractional perfect matching.
\end{theorem}
\begin{proof}
If $q_1(G)<\frac{2n\delta}{n-1}$, then it follows from  Lemma \ref{lemmm2} that $\alpha^{\prime}_{*}(G)>\frac{n-1}{2}$. By Lemma \ref{lemm0}, $2\alpha^{\prime}_*(G)$ is an integer, then $\alpha^{\prime}_{*}(G)=\frac{n}{2}$, which means that $G$ has a fractional perfect matching.
\end{proof}
 We now give a sufficient condition for the existence of a fractional perfect matching in a graph in terms of the signless Laplacian spectral radius of its complement.
\begin{theorem}
	Let $G$ be an $n$-vertex connected graph with minimum degree $\delta$ and $G^c$ be the complement of $G$.  If $q_1(G^c)<2\delta$, then $G$ has a fractional perfect matching.
\end{theorem}

\begin{proof}
	Assume to the contrary that $\alpha^{\prime}_*(G)<\frac{n}{2}$. By Lemma \ref{lemm0}, there exists a vertex set $S\subseteq V(G)$ such that $i(G-S)-|S|>0$. Denote by $A$ the set of isolated vertices in $G-S$. Note that the neighbours of each isolated vertex belong to $S$, then $|S|\geq\delta$, which implies that $|A|\geq|S|+1\geq\delta+1$. Since $G^c[A]$ is a clique, by Lemmas \ref{lemmm1} and \ref{k1}, we have $$q_1(G^c)\geq q_1(G^c[A])=2(|A|-1)=2\delta,$$
	a contradiction. This completes the proof.
\end{proof}
\begin{theorem}
	Let $G$ be an $n$-vertex connected graph with minimum degree $\delta$ and $G^c$ be the complement of $G$.  If $q_1(G^c)<2\delta+1$, then $G$ has a fractional perfect matching unless $G\cong H_1\vee H_2$, where $H_1$ is a $(\delta+1)$-independent set and $H_2$ is any graph of order $\delta$.
\end{theorem}

\begin{proof}
Suppose that $\alpha^{\prime}_*(G)<\frac{n}{2}$. By Lemma \ref{lemm0}, there exists a vertex set $S\subseteq V(G)$ such that $i(G-S)-|S|>0$. Let $A$ be the set of isolated vertices in $G-S$.  Then $|A|\geq|S|+1\geq\delta+1$. If $|A|\geq\delta+2$, then there is a clique of order $\delta+2$ in $G^c$ and thus  $q_1(G^c)\geq2(\delta+1)$, a contradiction. Furthermore, we have $|A|=|S|+1=\delta+1$. If $V(G)\ne A\cup S$, then there is a clique of order $\delta+2$ in $G^c$ and thus $q_1(G^c)\geq2(\delta+1)$, a contradiction. Hence, we have $V(G)=A\cup S$. Therefore, we have $G\cong H_1\vee H_2$. This completes the proof.
\end{proof}
For regular graphs, the authors in \cite{b,S} investigated the relations between the eigenvalues and the perfect matching. Here, we obtain the relations between the eigenvalues and the fractional  perfect matching for regular graphs.
 \begin{theorem}\label{ll}
 	Let $G$ be an $n$-vertex connected $k$-regular graph with eigenvalues $k=\lambda_1\geq\lambda_2\geq\cdots\geq\lambda_n$. If
 \begin{align*}
 & \hspace{3cm}\lambda_3\leq\left\{ 
 \begin{array}{ll}
 {k-1+\frac{3}{k+1},} &\textrm{if $k$ is even};\\
 {k-1+\frac{4}{k+2},} &\textrm{if $k$ is odd},\\
 \end{array} 
 \right. &
 \end{align*}
 then $G$ has a fractional  perfect matching.
 \end{theorem}

 \begin{proof}
 	Assume that  there exists not  a fractional perfect matching in $G$. Then $G$ has no perfect matching. Similar to the proof of the Theorem 4.8.9 in \cite{bre}, we can get a contradiction. 
 \end{proof}
 
 By Theorem \ref{ll}, we can get the following corollary immediately.
 
 \begin{corollary}
 	A regular graph with algebraic connectivity at least one has a fractional perfect matching.
 \end{corollary}
\section*{Acknowledgement(s)}
The authors would like to express their sincere gratitude to all the referees for their careful reading and insightful suggestions.

\end{document}